\documentclass[psamsfonts]{amsart}
\usepackage{amssymb}
\usepackage{mathtools,leftidx}
\usepackage{amscd, enumerate}
\usepackage[utf8]{inputenc}
\usepackage[all]{xy}
\usepackage{graphicx}
\newtheorem{thm}{Theorem}[section]
\newtheorem{prop}[thm]{Proposition}
\newtheorem{lem}[thm]{Lemma}
\newtheorem{cor}[thm]{Corollary}

\theoremstyle{definition}
\newtheorem{definition}[thm]{Definition}
\newtheorem{notation}[thm]{Notation}

\newtheorem{prob}[thm]{Problem}

\theoremstyle{remark}

\newtheorem{rem}[thm]{Remark}
\DeclareMathOperator*{\ModR}{Mod-R}
\DeclareMathOperator*{\SubLim}{SubLim}

\DeclareMathOperator*{\modR}{mod-R}
\DeclareMathOperator*{\ext}{Ext}

\DeclareMathOperator*{\tor}{Tor}
\DeclareMathOperator*{\Hom}{Hom}
\DeclareMathOperator*{\Ker}{Ker}
\DeclareMathOperator*{\spec}{Spec}

\DeclareMathOperator{\gen}{Gen}

\DeclareMathOperator*{\Div}{-Div}
\DeclareMathOperator*{\ann}{Ann}
\DeclareMathOperator*{\ass}{Ass}
\DeclareMathOperator*{\Span}{Span}
\DeclareMathOperator*{\vass}{VAss}
\DeclareMathOperator*{\pd}{pd}
\DeclareMathOperator*{\id}{id}
\DeclareMathOperator*{\tr}{Tr}
\DeclareMathOperator*{\ctr}{tr}
\DeclareMathOperator*{\supp}{Supp}
\DeclareMathOperator*{\steq}{\stackrel{st}{\simeq}}

\begin{document}

\title{One-tilting classes and modules over commutative rings}
\author{Michal Hrbek}
\address{Charles University \\
          Faculty of Mathematics and Physics\\
          Department of Algebra \\
          Sokolovsk\'a 83\\
          186 75 Praha 8\\
          Czech Republic}
		  \email{hrbmich@gmail.com}
 
\thanks{The author is partially supported by the Grant Agency of the Czech Republic under the grant no. 14-15479S and by the project SVV-2015-260227 of the Charles University in Prague.} 

\date{\today}

\begin{abstract}
		We classify 1-tilting classes over an arbitrary commutative ring. As a consequence, we classify all resolving subcategories of finitely presented modules of projective dimension at most 1. Both these collections are in 1-1 correspondence with faithful Gabriel topologies of finite type, or equivalently, with Thomason subsets of the spectrum avoiding a set of primes associated in a specific way to the ring. We also provide a generalization of the classical Fuchs and Salce tilting modules, and classify the equivalence classes of all 1-tilting modules. Finally we characterize the cases when tilting modules arise from perfect localizations.
\end{abstract}

\maketitle
\section{Introduction}
The classification of tilting classes and modules was done gradually, starting with abelian groups (\cite{GT1}), then small Dedekind domains, first assuming V=L (\cite{TW},\cite{TWC}), and then in ZFC (\cite{BET}), for Prüfer domains (\cite{BTC}), and almost perfect domains (\cite{APD}). Recently, in \cite{CN} the authors classified tilting classes of a commutative noetherian ring in terms of finite sequences of subsets of the Zariski spectrum of $R$. In particular, they proved that $1$-tilting classes correspond bijectively to specialization closed subsets of $\spec(R)$ that do not contain associated primes of $R$. We generalize this result to arbitrary commutative rings by showing that there is a one-to-one correspondence between 1-tilting classes and Thomason subsets of $\spec(R)$ that avoid primes ``associated'' to $R$ in certain sense. Thomason subsets of the spectrum coincide with specialization closed subsets in the noetherian case, and seem to be the correct generalization in various classification theorems. The prime example of this phenomenon is the classification of compactly generated localizing subcategories of the unbounded derived category of $R$ done first by Neeman for noetherian rings and then in general by Thomason (\cite{TH}). 

As in the noetherian case in (\cite{CN}), we start working in the dual setting of cotilting classes. Even though there is an explicit duality between tilting modules and cotilting modules of cofinite type, the one way nature of the duality makes the tilting side harder to approach. For example, cotilting modules over commutative noetherian case are described in \cite{CNC}, but tilting modules were described only for special classes of noetherian rings. The crucial step in our approach is to show that a $1$-cotilting class is of cofinite type if and only if it is closed under injective envelopes (Corollary~\ref{C30}).

Alternatively, 1-tilting classes over a commutative ring $R$ correspond bijectively to faithful finitely generated Gabriel topologies over $R$. From this point of view, our classification generalizes directly results for Prüfer domains from \cite{BTC}. If $R$ is not semihereditary, one has to replace the cyclic generators of the hereditary torsion class by their Auslander-Bridger transposes in order to describe the resolving subcategories of finitely presented modules of projective dimension at most $1$. In the second part of the paper, we use this idea and construct an associated tilting module for each 1-tilting class over a commutative ring. This construction generalizes the Fuchs and Salce tilting modules introduced by Facchini, Fuchs-Salce, and Salce (\cite{F}, \cite{FS}, \cite{S}) from multiplicative sets over a domain and finitely generated Gabriel topology over a Prüfer domain to general faithful finitely generated Gabriel topology over a commutative ring.

In the rest of the second section we use the ``minimality'' of the constructed 1-tilting modules and provide an elementary proof of the commutative version of the recently solved Saorín's problem (\cite{PP}). Finally, in the last section we show that a 1-tilting module arises from a perfect localization if and only if the associated Gabriel topology is perfect and the induced perfect localization has projective dimension 1.
\section{Preliminaries}
\subsection{Basic notation and cotorsion pairs}
Given an (associative, unital) ring $R$, we denote by $\ModR$ the category of all right $R$-modules and by $\modR$ the full subcategory of $\ModR$ consisting of all finitely presented right $R$-modules.

	For a class of right $R$-modules $\mathcal{S}$, we will use the following notation:
	$$\textstyle\mathcal{S}^\perp=\{M \in \ModR \mid \ext_R^1(S,M)=0 \text{ for all $S \in \mathcal{S}$}\},$$
	$$\textstyle^\perp\mathcal{S}=\{M \in \ModR \mid \ext_R^1(M,S)=0 \text{ for all $S \in \mathcal{S}$}\}.$$

	Similarly, if $\mathcal{S}$ is a class of left $R$-modules we let:
	$$\textstyle\mathcal{S}^\intercal=\{M \in \ModR \mid \tor_R^1(M,S)=0 \text{ for all $S \in \mathcal{S}$}\}.$$

	Given a class $\mathcal{S}$, a chain of submodules of an $R$-module $M$ 
	$$0=M_0 \subseteq M_1 \subseteq \cdots \subseteq M_\alpha \subseteq M_{\alpha+1} \subseteq \cdots \subseteq M_\lambda=M$$
	indexed by ordinal $\lambda+1$, with the property that $M_\beta=\bigcup_{\alpha<\beta}M_\alpha$ for each limit ordinal $\beta \leq \lambda$ and $M_{\alpha+1}/M_\alpha$ is isomorphic to some module from $\mathcal{S}$ for each $\alpha < \lambda$, is called an $\mathcal{S}$\emph{-filtration} of $M$. We say that $M$ is $\mathcal{S}$\emph{-filtered} if it possesses an $\mathcal{S}$-filtration.

	A couple of full subcategories $(\mathcal{A},\mathcal{B})$ of $\ModR$ is called a \emph{cotorsion pair} provided that $\mathcal{A}={}^{\perp}\mathcal{B}$ and $\mathcal{B}=\mathcal{A}^{\perp}$. Given a class $\mathcal{S}$ of modules, the cotorsion pair $(^{\perp}(\mathcal{S}^{\perp}),\mathcal{S}^{\perp})$ is \emph{generated} by $\mathcal{S}$. The following important result about cotorsion pairs generated by sets of modules will be used freely throughout the paper.
	\begin{lem}
			(\cite[Corollary 6.14]{GT}) Let $\mathcal{S}$ be a set of modules and $(\mathcal{A},\mathcal{B})$ the cotorsion pair generated by $\mathcal{S}$. Then $\mathcal{A}$ consists precisely of all direct summands of all $\mathcal{S}$-filtered modules.
	\end{lem}
\subsection{Gabriel topologies, torsion pairs, and divisibility}
Given a right ideal $I$ and an element $t$ of a ring $R$, we denote $(I : t)=\{r \in R \mid tr \in I\}$.
\begin{definition}
	A filter $\mathcal{G}$ of right ideals of $R$ is called a \emph{Gabriel topology} provided that:
	\begin{itemize}
			\item if $I \in \mathcal{G}$ and $t \in R$, then $(I : t) \in \mathcal{G}$,
			\item if $J$ is a right ideal and $I \in \mathcal{G}$ is such that $(J : t) \in \mathcal{G}$ for any $t \in I$, then $J \in \mathcal{G}$.
	\end{itemize}
	A Gabriel topology is \emph{finitely generated} if it has a basis of finitely generated right ideals. A right ideal $I$ of $R$ is \emph{faithful} if $\ann I=0$ (if $R$ is commutative, this is equivalent to $\Hom_R(R/I,R)=0$). We say that a Gabriel topology is \emph{faithful} if it has a basis consisting of faithful ideals (and thus all ideals in $\mathcal{G}$ are faithful).
\end{definition}
	There is an easier description of finitely generated Gabriel topologies over commutative rings.
	\begin{lem}
			\label{L233}
		Suppose $R$ is commutative. A filter $\mathcal{G}$ of ideals of $R$ with a basis of finitely generated ideals is a (finitely generated) Gabriel topology iff it is closed under ideal products.
	\end{lem}
	\begin{proof}
			If $\mathcal{G}$ is a Gabriel topology, then it is closed under products, since for any $i \in I$ we have $(IJ : i) \supseteq J$, and thus $IJ \in \mathcal{G}$, provided that $I,J \in \mathcal{G}$. Suppose that $\mathcal{G}$ is closed under products. Let $I \in \mathcal{G}$ and $t \in R$. Since $R$ is commutative, $I \subseteq (I : t)$ and thus the latter ideal is in $\mathcal{G}$. Let now $J$ be any ideal and $I \in \mathcal{G}$ such that $(J : t) \in \mathcal{G}$ for each $t \in I$. We want to show that $J \in \mathcal{G}$. By the hypothesis, we can assume that $I$ is finitely generated, say with a generating set $\{i_1,i_2,\ldots,i_n\}$. We have $I_k=(J : i_k) \in \mathcal{G}$ for all $k=1,2,\ldots,n$. It follows that $II_1I_2\cdots I_k \subseteq J$, and thus $J \in \mathcal{G}$, as claimed.
	\end{proof}

	We say that a pair of full subcategories $(\mathcal{T},\mathcal{F})$ of $\ModR$ is a \emph{torsion pair} provided that $\mathcal{T}=\{M \in \ModR \mid \Hom_R(M,F)=0 \text{ for all $F \in \mathcal{F}$}\}$ and $\mathcal{F}=\{M \in \ModR \mid \Hom_R(T,M)=0 \text{ for all $T \in \mathcal{T}$}\}$. The class $\mathcal{T}$ (resp. $\mathcal{F}$) is called a \emph{torsion} (resp. \emph{torsion-free}) class. A class of modules fits into a torsion pair as a torsion (resp. torsion-free) class iff it is closed under extensions, direct sums, and homomorphic images (resp. under extensions, direct products, and submodules). Such torsion pair is said to be:
\begin{itemize}
		\item \emph{hereditary} provided that $\mathcal{T}$ is closed under submodules (or, equivalently, $\mathcal{F}$ is closed under injective envelopes),
		\item \emph{faithful} provided that $R \in \mathcal{F}$,
		\item \emph{of finite type} provided that $\mathcal{F}$ is closed under direct limits.
\end{itemize}

With any torsion pair $(\mathcal{T},\mathcal{F})$ in $\ModR$ there is an associated idempotent subfunctor $t$ on $\ModR$ called the \emph{torsion radical}, defined by the property that for any module $M$, we have $t(M) \in \mathcal{T}$ and $M/t(M) \in \mathcal{F}$. It is easy to see that $(\mathcal{T},\mathcal{F})$ is of finite type iff $t$ commutes with direct limits. The following observation will be useful in characterizing cotilting torsion-free classes of cofinite type.

\begin{lem}
\label{L77}
	A hereditary torsion pair $(\mathcal{T},\mathcal{F})$ is of finite type iff there is a set $\mathcal{S}$ of finitely presented modules such that $\mathcal{F}=\{M \in \ModR \mid \Hom_R(S,M)=0 \text{ for all $S \in \mathcal{S}$}\}$.
\end{lem}
\begin{proof}
	The if-part follows from the fact that $\Hom_R(S,-)$ commutes with direct limit for any finitely presented module $S$.

	Let us prove the other implication. Since the pair is hereditary, there is a set $\mathcal{E}$ of finitely generated modules such that $\mathcal{F}=\Ker \Hom_R(\mathcal{E},-)$ (e.g. the set of all cyclic modules from $\mathcal{T}$). We are left to show that we can find such set consisting of finitely presented modules. Fix $M \in \mathcal{E}$. Let $0 \rightarrow K \rightarrow R^n \rightarrow M \rightarrow 0$ be a free presentation of $M$. We can write $K$ as a directed union $K=\bigcup_{i \in I}K_i$ of its finitely generated submodules. Then $M$ is a direct limit of finitely presented modules $R^n/K_i, i \in I$ in a way that all the maps of this direct system are projections. Since the torsion radical $t$ of the torsion pair $(\mathcal{T},\mathcal{F})$ commutes with direct limits, we have that $M=t(M)=t(\varinjlim_I R^n/K_i)=\varinjlim_I t(R^n/K_i)$. Let $J_i, i \in I$ be submodules of $R^n$ containing $K_i$ such that the torsion-free part of $R^n/K_i$ is isomorphic to $R^n/J_i$ for each $i \in I$. Since $\varinjlim_I R^n/J_i$ is isomorphic to the torsion-free part of $M$, it is zero, and thus $\varinjlim_I J_i=\bigcup_I J_i=R^n$. As $R^n$ is finitely generated, there is $k \in I$ with $J_k=R^n$. It follows that $R/K_k$ is in $\mathcal{T}$, and thus $M$ is a direct limit of finitely presented modules $R^n/K_i, i \geq k$, which all belong to $\mathcal{T}$, because the directed system consisted of projections. Put $\mathcal{S}_M=\{R^n/K_i, i \geq k\}$. Because $\mathcal{S}_M \subseteq \mathcal{T}$ generates $M$, we infer that $\mathcal{T}=\Ker \Hom_R(\mathcal{E} \setminus \{M\} \cup \mathcal{S}_M,-)$. 

	Constructing the set of finitely presented modules $\mathcal{S}_M$ for each $M \in \mathcal{E}$ and putting $\mathcal{S}=\bigcup_{M \in \mathcal{E}}\mathcal{S}_M$, we infer that $\mathcal{T}=\Ker \Hom_R(\mathcal{S},-)$ as desired.
\end{proof}
Given a Gabriel topology $\mathcal{G}$, there is a hereditary torsion pair induced by $\mathcal{G}$ with the torsion class $\{M \mid \ann(m) \in \mathcal{G} \text{ for all $m \in M$}\}$. Also, there is another torsion pair (usually not hereditary) with the torsion class $\{M \in \ModR \mid M=MI \text{ for all $I \in \mathcal{G}$}\}$.
\begin{thm}
		(\cite[\S VI.Theorem 5.1]{BS}) Let $R$ be a ring $R$. There is a 1-1 correspondence between hereditary torsion pairs $(\mathcal{T},\mathcal{F})$ in $\ModR$ and Gabriel topologies $\mathcal{G}$ given by
	$$\mathcal{T} \mapsto \{I \text{ right ideal} \mid R/I \in \mathcal{T}\},$$
	$$\mathcal{G} \mapsto \{M \in \ModR \mid \ann(m) \in \mathcal{G} \text{ for all $m \in M$}\}.$$
\end{thm}
\begin{notation}
	Given a set of (right) ideals $\mathcal{I}$, we denote by $\mathcal{I}\Div$ the class of all $\mathcal{I}$-divisible right modules, that is, the class $\{M \in \ModR \mid M=MI \text{ for all $I \in \mathcal{I}$}\}$.
\end{notation}
\subsection{Prime spectrum}
Given a commutative ring $R$, we denote by $\spec R$ the prime spectrum of $R$. Set $\spec R$ is endowed with the Zariski topology, i.e. the topology with closed sets being the sets of form
$$V(I)=\{\mathfrak{p} \in \spec R \mid I \subseteq \mathfrak{p}\},$$
for some ideal $I$ of $R$. Following the work of Thomason (\cite{TH}), we say that a subset of $\spec R$ is \emph{Thomason} if it is a union of sets $V(I)$ with $I$ being finitely generated (equivalently, if it is a union of Zariski closed sets with quasi-compact complements). It is well-known that Thomason subsets of $\spec R$ correspond bijectively to finitely generated Gabriel topologies (by assigning to a finitely generated Gabriel topology the set of all primes contained in it). We will prove a ''faithful`` version of this fact in Theorem~\ref{T01} for convenience.

For any $M \in \ModR$, symbol $\ass M$ stands for the set of all associated primes of $M$, that is, all primes $\mathfrak{p} \in \spec R$ such that $R/\mathfrak{p}$ embeds into $M$. Similarly, for a subclass $\mathcal{C}$ of $\ModR$ we fix a notation $\ass \mathcal{C}=\bigcup_{M \in \mathcal{C}}\ass M$.

We denote the localization of $R$ at prime $\mathfrak{p}$ by $R_\mathfrak{p}$. For any $M \in \ModR$ we put $M_\mathfrak{p}=M \otimes_R R_\mathfrak{p}$. The set of all primes $\mathfrak{p}$ such that $M_\mathfrak{p}$ is non-zero is called the \emph{support} of $M$ and denoted by $\supp M$. It is well-known that $\supp M=\{\mathfrak{p} \in \spec R \mid \ann M \subseteq \mathfrak{p}\}$ provided that $M$ is finitely generated. 

\subsection{Tilting and cotilting}
We use the following definition of an (infinitely generated) right $1$-tilting module over an arbitrary ring $R$ (\cite{TC},\cite{AC}).
	\begin{definition}
		An $R$-module $T$ is said to be \emph{1-tilting} if
		\begin{itemize}
			\item $\pd T \leq 1$,
			\item $\ext_R^1(T,T^{(X)})=0$ for any set $X$,
			\item there is an exact sequence $0 \rightarrow R \rightarrow T_0 \rightarrow T_1 \rightarrow 0$, where $T_i$ is a direct summand of a direct sum of copies of $T$ for each $i=0,1$.
		\end{itemize}
		The class $\mathcal{T}=T^\perp$ is called a \emph{1-tilting class}, and the induced cotorsion pair $(\mathcal{A},\mathcal{T})$ a \emph{1-tilting cotorsion pair}. Two 1-tilting modules $T$ and $T'$ are said to be equivalent if $T^\perp=T'^\perp$.
	\end{definition}
	\begin{rem}
			Given a module $M$, denote by $\gen(M)$ the class of all homomorphic images of direct sums of copies of $M$. We remark that module $T$ is 1-tilting if and only if $\gen(T)=T^\perp$ (see \cite[Lemma 14.2]{GT}), providing an easier alternative definition. In particular, note that a $1$-tilting class is a torsion class. Also, a class $\mathcal{T}$ is 1-tilting if and only if it is a special preenveloping torsion class (\cite[Theorem 14.4]{GT}).
	\end{rem}
	Classical tilting theory of artin algebras focuses on finitely presented tilting modules, which is in stark contrast with the commutative setting, where only infinitely generated ones are interesting:
	\begin{lem}
			(\cite[Lemma 1.2]{TP}) Let $R$ be a commutative ring. Then any 1-tilting module equivalent to a finitely generated one is projective.
	\end{lem}
	On the other hand, infinitely generated tilting modules share a lot of properties of their classical finitely presented counterparts. In particular, they still serve as a generalization of progenerators from the classical Morita equivalence, as they induce equivalences of subcategories of module categories, or derived equivalences between triangulated subcategories of derived module categories (\cite{B2}, \cite{BMT}).

	Even though tilting modules over commutative rings are almost always infinitely generated, the tilting classes can be fully described in terms of the small module category $\modR$. Using the Small Object Argument, one can show that given a set $\mathcal{S}$ of finitely presented modules of projective dimension at most $1$, the class $\mathcal{S}^\perp$ is 1-tilting. Crucial results by Bazzoni-Herbera (and Bazzoni-Štovíček for general n-tilting classes) show that the converse is also true. We recall that a subcategory $\mathcal{S}$ of $\modR$ is \emph{resolving}, if all finitely generated projectives are contained in $\mathcal{S}$, and $\mathcal{S}$ is closed under extensions, direct summands, and syzygies.
	\begin{thm}
\label{T67}
			(\cite{BH}, \cite{BST}) Let $R$ be a ring. There is a 1-1 correspondence between 1-tilting classes $\mathcal{T}$ in $\ModR$, and resolving subcategories $\mathcal{S}$ of $\modR$ contained in $\{M \in \modR \mid \pd M \leq 1\}$. The corespondence is given by the assignments $\mathcal{S} \mapsto \mathcal{S}^\perp$ and $\mathcal{T} \mapsto (^\perp\mathcal{T}) \cap \modR$.
	\end{thm}
	\begin{definition}
			An $R$-module $C$ is said to be \emph{1-cotilting} if
			\begin{itemize}
			\item $\id C \leq 1$,
			\item $\ext_R^1(C^X,C)=0$ for any set $X$,
			\item there is an exact sequence $0 \rightarrow C_1 \rightarrow C_{0} \rightarrow W \rightarrow 0$, where $W$ is an injective cogenerator in $\ModR$, and $C_i$ is a direct summand of a direct product of copies of $C$ for each $i=0,1$.
			\end{itemize}
			The class $\mathcal{C}={}^\perp C$ is called a $\emph{1-cotilting class induced by C}$ and 1-cotilting modules $C,C'$ are said to be \emph{equivalent} if their induced cotilting classes coincide.
	\end{definition}
	Unlike tilting classes, 1-cotilting classes do not in general come from a set of finitely presented modules unless the ring is noetherian (see a counter-example due to Bazzoni in \cite[Proposition 4.5]{BTC}).
	\begin{definition}
		A $1$-cotilting class $\mathcal{C}$ is of $\emph{cofinite type}$ provided there is a set of finitely presented modules $\mathcal{S}$ of projective dimension at most $1$ such that $\mathcal{C}=\mathcal{S}^\intercal$.
	\end{definition}
	Given a $1$-tilting right module $T$, its character module $T^+=\Hom_{\mathbb{Z}}(T,\mathbb{Q}/\mathbb{Z})$ is a $1$-cotilting left $R$-module. Furthermore, if $\mathcal{S}$ is a subset of $\modR$ such that $T^\perp=\mathcal{S}^\perp$, then the cotilting class $^\perp(T^+)$ equals $\mathcal{S}^\intercal$, and thus is of cofinite type. In fact, the assigment $T \mapsto T^+$ induces a 1-1 correspondence between equivalence classes of 1-tilting right $R$-modules and equivalence classes of 1-cotilting left modules of cofinite type (meaning that the induced 1-cotilting class is of cofinite type). For details, see \cite[\S 15]{GT}.
\section{Tilting and cofinite-type cotilting classes}
\subsection{General formulas}
	We start with recalling the notion of transpose from \cite{AB}. Although this idea was originally used mostly in the artin algebra setting, it has proven useful in classifying tilting classes over commutative noetherian rings in \cite{CN}. In fact, it will serve the same purpose over a general commutative ring.
\begin{definition}
		Let $R$ be a ring and $M$ a finitely presented left $R$-module. Let $P_1 \xrightarrow{f} P_0 \rightarrow M \rightarrow 0$ be a presentation of $M$ with both $P_0$ and $P_1$ finitely generated projectives. We use the notation $(-)^*$ for the regular module duality functor $\Hom_R(-,R)$. The \emph{(Auslander-Bridger) transpose} of $M$ is obtained as the cokernel of the map of right $R$-modules $f^*: P_0^* \rightarrow P_1^*$. We denote the transpose by $\tr M$.
\end{definition}
\begin{rem}
		 It is important to note that the right $R$-module $\tr M$ is uniquely determined only up to \emph{stable equivalence}, that is, up to splitting off or adding a projective direct summand (\cite[\S 2.1]{AB}). We will use the notation $M \steq N$ for $M$ being stably equivalent to $N$.

		 However, there is a nice choice of a concrete representative module for $\tr M$ if $\pd_R M \leq 1$. Indeed, then $\ext_R^1(M,R) \steq \tr M$ (see Lemma~\ref{L03} below).
\end{rem}
We gather several well-known homological formulae for the transpose we will need later on, and reprove them in our setting for convenience.
\begin{lem}
	\emph{(\cite[Lemma 2.9]{CN})}\label{L02}
	Let $R$ be a ring, $M$ a non-zero left finitely presented $R$-module, such that $\Hom_R(M,R)=0$. Then:
	\begin{enumerate}
			\item[(i)] $\pd_R \tr M = 1$ and $\tr M$ is finitely presented,
			\item[(ii)] $\Hom_R(M,-)$ and $\tor_1^R(\tr M,-)$ are isomorphic functors,
			\item[(iii)] $\ext_R^1(\tr M,-)$ and $(- \otimes_R M)$ are isomorphic functors.
	\end{enumerate}
\end{lem}
\begin{proof}
	\begin{enumerate}
			\item[(i)] Since $M$ is finitely presented, there is a part of a projective resolution of $M$
				\begin{equation}
					\label{E01}
				P_{1} \rightarrow P_0 \rightarrow M \rightarrow 0,
				\end{equation}
				consisting of finitely generated projectives. Applying $(-)^*$ we get a complex
				$$0 \leftarrow \tr M \leftarrow P_{1}^* \leftarrow P_0^* \leftarrow 0,$$
				which is exact by our hypothesis on $M$, showing that $\pd\tr  M \leq 1$, and that $\tr M$ is finitely presented. If $\tr M$ was projective, then $M$ is projective, which together with $M^*=0$ implies that $M=0$, a contradiction. Hence, $\pd \tr M=1$.
		\item[(ii)] Let $N$ be a left $R$-module. By definition, $\ext_R^1(M,N)$ is the first homology of the complex obtained by applying $\Hom_R(-,N)$ on $(\ref{E01})$. We now use the natural isomorphism $\Hom_R(P,N) \simeq P^* \otimes_R N$ where $P$ is a finitely generated projective (see \cite[Proposition 20.6]{AF}) to infer the desired isomorphism.
		\item[(iii)] Analogous.
	\end{enumerate}
\end{proof}
A sort of converse for Lemma~\ref{L02} also holds, if we choose a representative of the transpose well enough. Unlike Lemma~\ref{L02}, this result does not generalize to higher projective dimension in a straightforward way.
\begin{lem}
		\label{L03}
		Let $S \in \modR$ be such that $\pd_R S \leq 1$. Put $S^\dagger=\ext_R^1(S,R)$. Then $S^\dagger$ is a finitely presented left $R$-module with $(S^\dagger)^*=0$, and $S^\dagger \steq \tr S$.
\end{lem}
\begin{proof}
	Let
	\begin{equation}
			\label{E02}
	0 \rightarrow P_1 \rightarrow P_0 \rightarrow S \rightarrow 0,
	\end{equation}
	be a projective resolution of $S$ consisting of finitely generated projectives. Applying $(-)^*$ we get an exact sequence
	$$0 \leftarrow S^\dagger \leftarrow P_{1}^* \leftarrow P_0^* \leftarrow S^* \leftarrow 0,$$
	showing that $S^\dagger$ is finitely presented, and by the definition $S^\dagger \steq \tr S$. Applying $(-)^*$ again we get back to the exact sequence (\ref{E02}), proving that $(S^\dagger)^*=0$.
\end{proof}
\begin{notation}
	We fix the notation $S^\dagger=\ext_R^1(S,R)$ for any $S \in \ModR$.
\end{notation}
Combining Lemma~\ref{L02} and Lemma~\ref{L03} we obtain:
\begin{cor}
		\label{C01}
	Let $\mathcal{S}$ be a set of finitely presented right $R$-modules of projective dimension 1. Let $\mathcal{T}=\mathcal{S}^\perp$ and $\mathcal{C}=\mathcal{S}^\intercal$ be the induced 1-tilting and 1-cotilting class of cofinite type. Then:
	\begin{itemize}
	\item $\mathcal{T}=\bigcap_{S \in \mathcal{S}} \Ker (- \otimes_R S^\dagger),$
	\item $\mathcal{C}=\bigcap_{S \in \mathcal{S}} \Ker \Hom_R(S^\dagger,-).$
	\end{itemize}
\end{cor}
\begin{proof}
		By Lemma~\ref{L03}, the module $S^\dagger$ satisfies $S^*=0$ and $\pd S \leq 1$ for each $S \in \mathcal{S}$. Therefore, we can use Lemma~\ref{L02}(3) to infer that $\bigcap_{S \in \mathcal{S}}\Ker \ext_R^1(\tr S^\dagger,-)=\bigcap_{S \in \mathcal{S}}\Ker (- \otimes_R S^\dagger)$. As $\tr S^\dagger \steq S$ by Lemma~\ref{L03}, this class is equal to $\mathcal{T}$ as desired. The formula for the cotilting class is derived analogously, using Lemma~\ref{L02}(2).
\end{proof}
\subsection{Commutative rings}
From now on, let $R$ be a {\bf commutative ring}. We begin by proving that any cofinite type 1-cotilting torsion pair is hereditary. Note that this in general fails for non-commutative rings\footnote{Counter-example (communicated to the author by Jan Šťovíček) can be obtained as follows. Let $A$ be a left hereditary right artinian ring such that the only projective injective right module is zero (e.g. the Kronecker algebra over a field). Then the class of all projective (equally, flat) right $A$-modules is equal to $(\operatorname{R-mod})^\intercal$, and thus is a 1-cotilting class of cofinite type not closed under injective envelopes.}. If $R$ is not noetherian, the classical theory of associated primes does not function well. Indeed, there can be non-zero modules with no associated primes\footnote{Easy example can be obtained as follows. Let $R$ be a valuation domain of Krull dimension 1 with idempotent radical (e.g. the ring of all Puiseux series over a field). Then it is an easy exercise to show that any cyclic module of form $R/rR$ for $r \in R$ non-zero has zero socle, and thus it has no associated primes.}. The following notion will prove useful in the cotilting setting.

\begin{definition}
		Given a class $\mathcal{C}$ of modules, let $\SubLim(\mathcal{C})$ denote the smallest (isomorphism-closed) subclass of $\ModR$ containing $\mathcal{C}$ closed under direct limits and submodules.

		We say that a prime $\mathfrak{p}$ is \emph{vaguely associated} to a module $M$ if $R/\mathfrak{p}$ is contained in $\SubLim(\{M\})$. Denote the set of all vaguely associated primes of $M$ by $\vass M$. 
\end{definition}
First, we note that this is indeed a generalization of the concept of associated primes over noetherian rings.
\begin{lem}
	Let $R$ be noetherian, then $\vass M = \ass M$ for any $R$-module $M$.
\end{lem}
\begin{proof}
		Let $\mathfrak{p} \in \vass M$ and let us show that $\mathfrak{p} \in \ass M$. By definition, we have that $R/\mathfrak{p} \in \SubLim(\{M\})$. Since taking submodules does not introduce any new associated primes, it is enough to show that whenever $L$ is a direct limit of a direct system $L_i, i \in I$ such that $\mathfrak{p} \not\in \ass L_i$ for each $i \in I$, then $\mathfrak{p} \not\in \ass L$. Using \cite[Corollary 2.9]{GT}, there is a pure epimorphism $\pi: \bigoplus_{i \in I} L_i \rightarrow L$. Since $R$ is noetherian, the module $R/\mathfrak{p}$ is finitely presented, and thus we can factorize the inclusion $R/\mathfrak{p} \xhookrightarrow{} L$ through $\pi$. Therefore, $R/\mathfrak{p} \in \ass \bigoplus_{i \in I}L_i$. This already shows that there is $i \in I$, such that $\mathfrak{p} \in L_i$, a contradiction. We showed that $\vass M \subseteq \ass M$; the inverse inclusion is trivially true.
\end{proof}
\begin{lem}
	\label{L01}
	If $M$ is non-zero, then $\vass M$ is non-empty. Also, $\vass M \subseteq \supp M$.
\end{lem}
\begin{proof}
		Define 
		$$X=\{I \text{ ideal} \mid R \neq I \text{ and } R/I \in \SubLim(\{M\})\}.$$ 
		Since $M$ is non-zero, $X$ is non-empty. We claim that $X$ is inductive (with respect to inclusion). Indeed, let $\mathfrak{c}$ be an increasing chain of ideals from $X$ and put $I=\bigcup \mathfrak{c}$. As $R \not\in \mathfrak{c}$, also $I \neq R$. The cyclic module $R/I$ can be obtained as a direct limit of the modules $R/J, J \in \mathfrak{c}$, from $\SubLim(\{M\})$. Then $R/I$ is an element of $\SubLim(\{M\})$, and thus $I \in X$, proving that $X$ is inductive.

		We can thus use Zorn's Lemma to find a maximal element $\mathfrak{p}$ of $X$, which is easily seen to be prime, hence $\mathfrak{p} \in \vass M$.

		Finally, any $\mathfrak{p} \in \vass(M)$ has to be in the support of $M$, because $R/\mathfrak{p} \in \SubLim(\{M\})$, and the localization functor $- \otimes_R R_\mathfrak{p}$ commutes with submodules and direct limits.				
\end{proof}
\begin{lem}
		For a finitely generated module $M$ and prime $\mathfrak{p}$, $\Hom_R(M, R/\mathfrak{p}) \neq 0$ iff $\mathfrak{p} \in \supp M$.
\end{lem}
\begin{proof}
	Suppose first that there is a non-zero map $f: M \rightarrow R/\mathfrak{p}$. If $\mathfrak{p} \not\in \supp M$, then $R/\mathfrak{p}$ contains a non-zero $R/\mathfrak{p}$-torsion submodule, a contradiction.

	Let $\mathfrak{p} \in \supp M$. Consider the quotient $M/\mathfrak{p}M$. Localizing at $\mathfrak{p}$ we obtain $M_\mathfrak{p}/(\mathfrak{p}_\mathfrak{p} M_\mathfrak{p})$. As $M_\mathfrak{p}$ is a non-zero finitely generated $R_\mathfrak{p}$-module, the latter quotient is non-zero by Nakayama. It follows that the torsion-free quotient of $M/\mathfrak{p}M$ (considered now as a module over the integral domain $R/\mathfrak{p}$) is non-zero. This module is well-known to embed into a finite product of $R/\mathfrak{p}$ (see \cite[Lemma 16.1]{GT}). Hence, $\Hom_R(M,R/\mathfrak{p})$ is non-zero as claimed.
\end{proof}

\begin{prop}
	\label{P01}
	Any 1-cotilting class of cofinite type is closed under taking injective envelopes. That is, any 1-cotilting torsion pair is hereditary.
\end{prop}
\begin{proof}
		Let $\mathcal{C}$ be a 1-cotilting class and let $C$ be a 1-cotilting module cogenerating $\mathcal{C}$. Using \cite[Lemma 1.3]{FM}, it is enough to show that $E(C) \in \mathcal{C}$. Since $\mathcal{C}$ is of cofinite type, there is a set $\mathcal{S}$ of finitely presented modules of projective dimension 1 such that $\mathcal{C}=\mathcal{S}^\intercal$. By Corollary~\ref{C01}, putting $\mathcal{E}=\{S^\dagger \mid S \in \mathcal{S}\}$ we get $\mathcal{C}=\bigcap_{M \in \mathcal{E}}\Ker \Hom_R(M,-)$. Suppose that there is a non-zero map $M \rightarrow E(C)$ for some $M \in \mathcal{E}$. Its image has to intersect $C$ non-trivially. It follows that there is an ideal $J \neq R$ containing $\ann M$ such that $R/J \in \mathcal{C}$. By Lemma~\ref{L01}, there is a prime $\mathfrak{p} \in \vass(R/J)$. Since $\mathcal{C}$ is closed under submodules and direct limits, we have that $R/\mathfrak{p} \in \mathcal{C}$. On the other hand, since $M$ is finitely generated and $J \subseteq \mathfrak{p}$, we have that $\mathfrak{p} \in \supp M$, and thus $\Hom_R(M,R/\mathfrak{p})\neq 0$ by Lemma~\ref{L02}. This is a contradiction.
\end{proof}
\begin{cor}
		\label{C20}
	Let $R$ be a commutative ring. Then 1-cotilting classes of cofinite type in $\ModR$ coincide with torsion-free classes of faithful hereditary torsion pairs of finite type.
\end{cor}
\begin{proof}
	Any 1-cotilting class of cofinite type is a torsion-free class of a faithful hereditary torsion-pair of finite type by Corollary~\ref{C01} and Proposition~\ref{P01}. 

A torsion-free class of a faithful hereditary torsion pair of finite type is of form $\Ker \Hom_R(\mathcal{S},-)$ for a set of finitely presented modules $\mathcal{S}$ by Lemma~\ref{L77}. Since the pair is faithful, we have $S^*=0$ for each $S \in \mathcal{S}$, and thus the torsion-free class equals $\{\tr S \mid S \in \mathcal{S}\}^\intercal$ by Lemma~\ref{L02}, proving that it is a 1-cotilting class of cofinite type.
\end{proof}
Given a module $M$, we let $\operatorname{Prod}(M)$ denote the class of all modules isomorphic to a direct summand of product of copies of $M$.
\begin{cor}
		\label{C30}
		Let $R$ be a commutative ring, and $\mathcal{C}$ a 1-cotilting class. Then the following conditions are equivalent:
		\begin{enumerate}
				\item $\mathcal{C}$ is of cofinite type,
				\item $\mathcal{C}$ is closed under injective envelopes,
				\item for any 1-cotilting module $C$ with $\mathcal{C}={}^\perp C$, we have $E(C) \in \operatorname{Prod}(C)$.
		\end{enumerate}
\end{cor}
\begin{proof}
		(1) $\Rightarrow$ (2): Proposition~\ref{P01}.

		(2) $\Rightarrow$ (3): Let $\mathcal{W}=\mathcal{C}^\perp$. Since $E(C)$ is injective, we have $E(C) \in \mathcal{C} \cap \mathcal{W}$, and $\mathcal{C} \cap \mathcal{W} = \operatorname{Prod}(C)$ by \cite[Lemma 15.4]{GT}.
		
		(3) $\Rightarrow$ (1): As $E(C) \in \mathcal{C}$, the class $\mathcal{C}$ is closed under injective envelopes by \cite[Lemma 1.3]{FM}. Then the induced torsion pair $(\mathcal{E},\mathcal{C})$ is faithful, hereditary, and of finite type, and thus $\mathcal{C}$ is of cofinite type by Corollary~\ref{C20}.
\end{proof}
Before classifying all 1-tilting classes, we distinguish the following two steps.
\begin{lem}
		\label{L06}
		Let $\mathcal{T}$ be a 1-tilting class and $J$ an ideal such that $M=JM$ for each $M \in \mathcal{T}$. Then there is a finitely generated ideal $I \subseteq J$ such that $M=IM$ for each $M \in \mathcal{T}$.
\end{lem}
\begin{proof}
		Let $T$ be a 1-tilting module such that $\mathcal{T}=T^\perp=\gen(T)$. Since $\mathcal{T}$ is closed under direct products, we have that $T^T=JT^T$. Let ${\bf t}=(t)_{t \in T} \in T^T$ be the sequence of all elements of $T$. Since ${\bf t} \in T^T=JT^T$, there is a finitely generated ideal $I \subseteq J$ such that ${\bf t} \in IT^T$. Looking at the canonical projections, we infer that $t \in IT$ for each $t \in T$, showing that $T=IT$. But since $T$ generates $\mathcal{T}$, this means that $\mathcal{T} \subseteq \{M \in \ModR \mid M=IM\}$ as claimed.
\end{proof}
\begin{lem}
		\label{L05}
	Let $S$ be a finitely presented module of projective dimension 1. Then there is a finitely generated ideal $I$ such that $S^\perp=\{I\}\Div=\{M \in \ModR \mid M=IM\}$.
\end{lem}
\begin{proof}
		We first show that there is an ideal $J$ such that $S^\perp=\{J\}\Div$. By Corollary~\ref{C01}, $\mathcal{T}=\Ker - \otimes_R S^\dagger$. Also let $\mathcal{C}=\Ker \Hom_R(S^\dagger,-)$ be the induced 1-cotilting class and let $C=T^+$ be the 1-cotilting module dual to $T$. Let us fix a filtration $0=M_0 \subseteq M_1 \subseteq \cdots \subseteq M_n=S^\dagger$ of $S^\dagger$ by cyclic modules, that is, such that there is an ideal $J_i$ with $M_{i+1}/M_i \simeq R/J_i$ for each $i=0,\ldots,n-1$. We have shown in Proposition~\ref{P01} that $E(C) \in \mathcal{C}$, and thus $\Hom_R(S^\dagger,E(C))=0$. Since the functor $\Hom_R(-,E(C))$ is exact, it follows that $\Hom_R(R/J_i,E(C))=0$ for each $i=1,\ldots,n$, and thus also $\Hom_R(R/J_i,C)=0$ for each $i=1,\ldots,n$. Using the standard isomorphism $(R/J_i \otimes_R T)^+ \simeq \Hom_R(R/J_i,T^+)$, we get that $(R/J_i \otimes_R T)^+=0$, and thus $R/J_i \otimes_R T=0$. In other words, $T=J_iT$ for each $i=1,\ldots,n$, and thus $T=JT$, where we put $J=J_1J_2\cdots J_n$. We have proved that $\mathcal{T} \subseteq \{J\}\Div$. The other inclusion follows easily, as $\mathcal{S}^\dagger$ is filtered by $\{R/J_i \mid i=1,\ldots,n\}$.

		Since $S^\perp$ is a 1-tilting class, by Lemma~\ref{L06} there is a finitely generated ideal $I \subseteq J$ such that $S^\perp \subseteq \{I\}\Div$. The latter inclusion must be an equality, because $\{I\}\Div \subseteq \{J\}\Div=S^\perp$.
\end{proof}
\subsection{Main theorem}
\begin{thm}
	\label{T01}
	Let $R$ be a commutative ring. There are bijections between the following collections:
	\begin{enumerate}
			\item 1-tilting classes $\mathcal{T}$,
			\item 1-cotilting classes of cofinite type $\mathcal{C}$,
			\item faithful finitely generated Gabriel topologies $\mathcal{G}$,
			\item Thomason subsets $X$ of $\spec(R) \setminus \vass(R)$,
			\item faithful hereditary torsion pairs $(\mathcal{E},\mathcal{F})$ of finite type in $\ModR$,
			\item resolving subcategories of $\modR$ consisting of modules of projective dimension at most 1.
	\end{enumerate}
	The bijections are given as follows:
	\begin{center}
	\begin{tabular}{ | l | l | }
			\hline
			Bijection & Formula \\
			\hline
			(1) $\rightarrow$ (2) & $\mathcal{T} \mapsto (^\perp\mathcal{T} \cap \modR)^\intercal$ \\
			(1) $\rightarrow$ (3) & $\Psi: \mathcal{T} \mapsto \{I \text{ ideal} \mid M=IM \text{ for all $M \in \mathcal{T}$}\}$ \\
			(3) $\rightarrow$ (1) & $\Phi: \mathcal{G} \mapsto \mathcal{G}\Div=(\bigoplus_{I \in \mathcal{G} \text{, $I$ f.g.}}\tr(R/I))^\perp$ \\
			(3) $\rightarrow$ (4) & $\Xi: \mathcal{G} \mapsto \mathcal{G} \cap \spec(R)$ \\
			(4) $\rightarrow$ (3) & $\Theta: X \mapsto  \{J \text{ ideal} \mid \exists I \subseteq J \text{ finitely generated such that $V(I) \subseteq X$}\}$ \\
			(5) $\rightarrow$ (2) & $(\mathcal{E},\mathcal{F}) \mapsto \mathcal{F}$ \\
			(2) $\rightarrow$ (4) & $\mathcal{C} \mapsto (\spec(R) \setminus \ass \mathcal{C})$ \\
			(3) $\rightarrow$ (2) & $\mathcal{G} \mapsto \{M \in \ModR \mid \ann(m) \not\in \mathcal{G} \text{ for all non-zero $m \in M$}\}$ \\
			(3) $\rightarrow$ (6) & $\mathcal{G} \mapsto \mathcal{S}=\{M \in \modR \mid M \text{ is isomorphic to a direct summand}$\\ 
								  & of a finitely $\{R\}\cup\{\tr(R/I) \mid I \in \mathcal{G} \text{ f.g.}\}$-filtered module$\}$ \\
			\hline
	\end{tabular}
	\end{center}
\end{thm}
\begin{proof}
		(1) $\leftrightarrow$ (2): Follows by Theorem~\ref{T67} and using the character duality (see the last paragraph of Section~2).

		(1) $\leftrightarrow$ (3): First let us prove that the prescribed maps $\Psi: \mathcal{T} \mapsto \mathcal{G}$ and $\Phi: \mathcal{G} \mapsto \mathcal{T}$ are well-defined. Let $\mathcal{T}$ be a 1-tilting class. By Lemma~\ref{L06}, whenever $J \in \Psi(\mathcal{T})$, there is a finitely generated ideal $I \subseteq J$ with $I \in \Psi(\mathcal{T})$. Also, any ideal in $\Psi(\mathcal{T})$ is faithful. Indeed, otherwise the special $\mathcal{T}$-preenvelope of $R$ would have a non-zero annihilator, which is not the case. As $\Psi(\mathcal{T})$ is evidently a filter closed under products, we infer from Lemma~\ref{L233} that it is a faithful finitely generated Gabriel topology. On the other hand, if $\mathcal{G}$ is a faithful finitely generated Gabriel topology with basis of finitely generated ideals $\mathcal{I}$, then $\Phi(\mathcal{G})=\{M \in \ModR \mid M \otimes_R R/I=0 \text{ for each $I \in \mathcal{I}$}\}$, and thus $\Phi(\mathcal{G})=\mathcal{E}^\perp$, where $\mathcal{E}=\{\tr R/I \mid I \in \mathcal{I}\}$ by Lemma~\ref{L02} (explicitly, we use the isomorphism of functors $- \otimes_R R/I \simeq \ext_R^1(\tr(R/I),-))$. This is a 1-tilting class by the same lemma.
		
		We need to prove that $\Psi$ and $\Phi$ are mutually inverse. Let $\mathcal{T}$ be a 1-tilting class. It is easy to see that $\mathcal{T} \subseteq \Phi(\Psi(\mathcal{T}))$.  Let $\mathcal{S}$ be a set of finitely presented modules of projective dimension 1 such that $\mathcal{T}=\mathcal{S}^\perp$. By Lemma~\ref{L05}, there is for each $S \in \mathcal{S}$ a (again, necessarily faithful) finitely generated ideal $I_S$ such that $S^\perp=\{M \in \ModR \mid M=I_SM\}$. Put $\mathcal{I}=\{I_S \mid S \in \mathcal{S}\}$. It follows that $\mathcal{T}=\bigcap_{S \in \mathcal{S}}S^\perp=\mathcal{I}\Div$. Then $\mathcal{I} \subseteq \Psi(\mathcal{T})$, and therefore $\Phi(\Psi(\mathcal{T})) \subseteq \mathcal{T}$, proving that $\Phi(\Psi(\mathcal{T}))=\mathcal{T}$.

		Finally, let $\mathcal{G}$ be a faithful finitely generated Gabriel topology and let $\mathcal{I}$ be some basis of $\mathcal{G}$ of finitely generated ideals. Let $J \in \Phi(\Psi(\mathcal{G}))$ be a finitely generated ideal. Denote by $(\mathcal{A},\mathcal{T})$ the tilting cotorsion pair $(^\perp\Psi(\mathcal{G}),\Psi(\mathcal{G}))$. Note that since $\mathcal{T}=\mathcal{I}\Div$, this cotorsion pair is generated (in the sense of \cite[Definition 5.15]{GT}) by the set $\mathcal{S}=\{\tr(R/I) \mid I \in \mathcal{I}\}$. Since $\mathcal{T} \subseteq \{M \in \ModR \mid M=JM\}$, we have that $\tr(R/J)\footnote{The stable equivalence representative choices do not matter in this argument.} \in \mathcal{A}$. By \cite[Corollary 6.14]{GT} and the Hill Lemma (\cite[Theorem 7.10]{GT}), we infer that $\tr(R/J)$ is a direct summand of a module $N$ possessing a finite $\mathcal{S} \cup \{R\}$-filtration. 

		Therefore, there is a filtration $0=N_0 \subseteq N_1 \subseteq \cdots \subseteq N_l=N$ with $N_{i+1}/N_i \in \mathcal{S} \cup \{R\}$ for each $i=0,1,\ldots,l-1$. Let us apply the functor $\ext_R^1(-,R)=(-)^\dagger$ to this filtration. Since all modules in $\mathcal{S} \cup \{R\}$ have projective dimension at most 1, this functor will act as a right exact functor on this filtration. As the filtration of $N$ was finite, we obtain a filtration $0=M_0 \subseteq M_1 \subseteq \cdots \subseteq M_l = N^\dagger$ such that $M_{i+1}/M_i$ is isomorphic to a homomorphic image of $X^\dagger$ for some $X \in \mathcal{S} \cup \{R\}$ for each $i=0,1,\ldots,l-1$.

		Since $\pd(\tr(R/I))=1$ for any $I \in \mathcal{I}$, we can apply Lemma~\ref{L03} in order to see that $\tr(R/I)^\dagger \steq \tr \tr(R/I)$, and that $(\tr(R/I)^\dagger)^*=0$. The only possibility is that $\tr(R/I)^\dagger \simeq R/I$. As $R^\dagger=0$, we conclude that $N^\dagger$ admits a filtration $0=M_0' \subseteq M_1' \subseteq \cdots \subseteq M'_k=N^\dagger$ such that $M'_{i+1}/M'_i \simeq R/L_i$ for an ideal $L_i$ containing some ideal $I_i \in \mathcal{I}$ for each $i=0,1,\ldots,k-1$. Put $L=L_0L_1\cdots L_{k-1}$. Then $L \subseteq \ann N^\dagger$, and as $I_0I_1\cdots I_{k-1} \subseteq L$, we have that $L \in \mathcal{G}$. But $R/J \simeq \tr(R/J)^\dagger$ is a direct summand of $N^\dagger$, whence $L \subseteq \ann N^\dagger \subseteq \ann(R/J)=J$. Therefore, $J \in \mathcal{G}$, proving that $\Phi(\Psi(\mathcal{G}))=\mathcal{G}$.

		(3) $\leftrightarrow$ (4): Let us again first prove that prescribed maps $\Xi: \mathcal{G} \mapsto X$ and $\Theta: X \mapsto \mathcal{G}$ are well-defined. Let $\mathcal{G}$ be a faithful finitely generated Gabriel topology with basis $\mathcal{I}$ of finitely generated ideals. Then $\Xi(\mathcal{G})$ is equal to $\bigcup_{I \in \mathcal{I}}V(I)$, and therefore is a Thomason set. Suppose that there is $\mathfrak{p} \in \Xi(\mathcal{G}) \cap \vass(R)$. Let $\mathcal{C}$ be the 1-cotilting class of cofinite type associated to $\Phi(\mathcal{G})$. Since $\mathfrak{p} \in \vass(R)$, and $R \in \mathcal{C}$, we have that $R/\mathfrak{p} \in \mathcal{C}$. But this is a contradiction, because $\mathcal{C}=\bigcap_{J \in \mathcal{G}}\Ker \Hom_R(R/J,-)$ by the previous bijection and Lemma~\ref{L03} and Corollary~\ref{C01}. 

		Let $X$ be a Thomason subset of $\spec(R) \setminus \vass(R)$. It is easy to see that $\Theta(X)$ is a finitely generated Gabriel topology. Suppose that there is an ideal $I \in \Theta(X)$ and a non-zero map $R/I \rightarrow R$. Then there is $\mathfrak{p} \in \vass(R)$ such that $I \subseteq \mathfrak{p}$, and therefore $\mathfrak{p} \in \Theta(X)$. But then $\mathfrak{p} \in X$, a contradiction.

		Now we prove that $\Xi$ and $\Theta$ are mutually inverse. That $\Xi(\Theta(X))=X$ is easy to see. Let us show that $\Theta(\Xi(\mathcal{G}))=\mathcal{G}$. Clearly $\mathcal{G} \subseteq \Theta(\Xi(\mathcal{G}))$. Suppose that there is $I \in \Theta(\Xi(\mathcal{G})) \setminus \mathcal{G}$. Since $\mathcal{G}$ has a basis of finitely generated ideals, by Zorn's Lemma there is a maximal ideal with this property, let $I'$ be maximal such. Then $I'$ is necessarily prime. Since $\Theta(\Xi(\mathcal{G})) \cap \spec(R)=\mathcal{G} \cap \spec(R)$, we arrived at a contradiction. 

		(3) $\rightarrow$ (2): Correctness of this bijection follows from Lemma~\ref{L03} and Corollary~\ref{C01}. Indeed, the cotilting class dual to the tilting class $\Phi(\mathcal{G})$ is equal to $\bigcap_{I \in \mathcal{G}}\Ker \Hom_R(R/I,-)$.

		(2) $\rightarrow$ (4): Using the already established bijections, and that $\mathcal{C}$ is closed under submodules and direct limits, it is enough to show that $\mathcal{C}=\{M \in \ModR \mid \vass(M) \cap \mathcal{G} = \emptyset\}$, where $\mathcal{G}$ is the finitely generated Gabriel topology such that $\mathcal{C}=\bigcap_{I \in \mathcal{G}}\Ker \Hom_R(R/I,-)$. It is easily seen that $\mathfrak{p} \not\in \vass(M)$ for any prime $\mathfrak{p} \in \mathcal{G}$ and $M \in \mathcal{C}$. To prove the converse, suppose that $\vass(M) \cap \mathcal{G} = \emptyset$. If there was a non-zero map in $\Hom_R(R/I,M)$ with $I \in \mathcal{G}$, there would exist by Lemma~\ref{L01} a prime ideal $\mathfrak{p} \in \vass(M)$ such that $I \subseteq \mathfrak{p}$ (see the proof of Proposition~\ref{P01}), a contradiction. Therefore, we can conclude that $M \in \mathcal{C}$.

		(5) $\rightarrow$ (2): Direct consequence of Corollary~\ref{C20}.

		(3) $\rightarrow$ (6): By Theorem~\ref{T67}, there is a 1-1 correspondence between 1-tilting classes and resolving subcategories of projective dimension at most 1 given by $\mathcal{T} \mapsto \mathcal{S}=({}^\perp \mathcal{T}) \cap \modR$. If $\mathcal{G}$ is a Gabriel topology with $\mathcal{G}=\Psi(\mathcal{T})$, we know from above that the cotorsion pair $({}^\perp \mathcal{T},\mathcal{T})$ is generated by the set $\{R\}\cup\{\tr(R/I) \mid I \in \mathcal{G} \text{ f.g.}\}$. Then $\mathcal{S}=({}^\perp \mathcal{T}) \cap \modR$ has the desired form, and we established the correspondence.
\end{proof}
\section{Tilting modules}
\subsection{Fuchs-Salce tilting modules}
In the previous part we have proved that $1$-tilting classes coincide with the classes of all modules divisible by all ideals of a faithful finitely generated Gabriel topology. The purpose of this section is to construct $1$-tilting modules generating those classes, and hence classify all 1-tilting modules over commutative rings up to equivalence. Of course we can always construct such modules using the Small Object Argument (see \cite{ET} or \cite[Theorem 6.11, Remark 13.47]{GT}). However, the following construction is ``minimal'' in the sense that the resulting module has a filtration of length only $\omega$ by direct sums of finitely presented modules. Also, the explicit contruction allows for direct computations, as we will demonstrate in the next second subsection.

The following construction generalizes the tilting modules generating the class of all divisible modules over a domain due to Fachini (\cite{F}), of all $S$-divisible modules for a multiplicative set $S$ due to Fuchs-Salce (\cite{FS}), and of all $\mathcal{F}$-divisible modules for a finitely generated Gabriel topology $\mathcal{F}$ over a Prüfer domain due to Salce (\cite{S}).

\begin{definition}
\label{D01}
Let $R$ be a commutative ring, and let $\mathcal{I}$ be a set of faithful finitely generated ideals of $R$. For each $I \in \mathcal{I}$ fix a finite set of generators $\{x^I_1,x^I_2,\ldots,x^I_{n_I}\}$ of $I$. Let $\Lambda$ denote the set consisting of all finite sequences of pairs of the form $(I,k)$, where $I \in \mathcal{I}$, and $k < n_I$ (including the empty sequence denoted by $w$). Let $F$ be a free $R$-module with the basis $\Lambda$. Given two sequences $\lambda, \lambda' \in \Lambda$, we denote their concatenation by $\lambda \sqcup \lambda'$. In particular, symbol $\lambda \sqcup (I,k)$ will stand for appending pair $(I,k)$ to sequence $\lambda \in \Lambda$.

			Define a submodule $G$ of $F$ as the span of all elements of the form 
			$$\lambda - \sum_{k \in n_I} x^I_k(\lambda\sqcup(I,k)),$$
			for each $\lambda \in \Lambda$ and $I \in \mathcal{I}$. Put $M_\mathcal{I}=F/G$. Let us call the module $\delta_\mathcal{I}=M_\mathcal{I} \oplus M_\mathcal{I}/\Span(w)$ the \emph{Fuchs-Salce tilting module}.

			If $\mathcal{G}$ is a faithful finitely generated Gabriel topology, we will abuse the notation by writing $\delta_\mathcal{G}$ instead of $\delta_\mathcal{I}$, where $\mathcal{I}$ is the set of all finitely generated ideals from $\mathcal{G}$.
	\end{definition}
	\begin{figure}[h]
			\includegraphics[width=360pt]{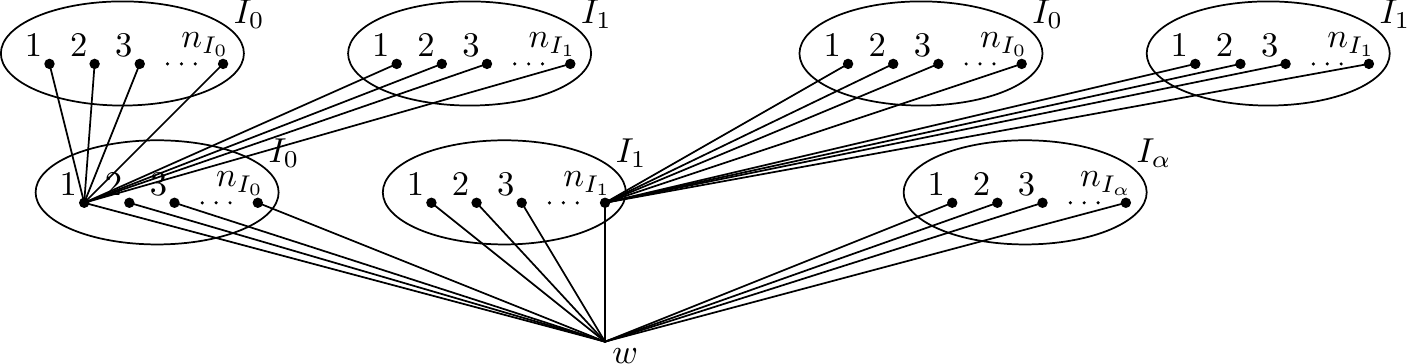}
			\caption{Construction of the Fuchs-Salce tilting module}
	Picture illustrates the first three levels of the homogeneous tree $\Lambda$ from Definition~\ref{D01}. The basis of the module $M_\mathcal{I}$ consists of all vertices of this tree. For each ``bubble'' we add one relation identifying the parent vertex with the linear combination of the vertices in the bubble with scalar coefficients being the chosen generators $x_1^I,x_2^I,\ldots,x_{n_I}^I$ of the ideal $I$.
	\end{figure}
	We fix a concrete representative in the stable equivalence class $\tr(R/I)$:
	\begin{notation}
		In the setting as in Definition~\ref{D01}, we define for each $I \in \mathcal{I}$ a module $\ctr(R/I)=R^{n_I}/(x^I_1,x^I_2,\ldots,x^I_{n_I})R$. Note that $\ctr(R/I) \steq \tr(R/I)$.
	\end{notation}
	\begin{prop}
			\label{P21}
			The module $\delta_\mathcal{I}$ defined above is a 1-tilting module generating the 1-tilting class $\mathcal{I}\Div$.
	\end{prop}
	\begin{proof}
			Put $\mathcal{T}=\mathcal{I}\Div$. First note that by the hypothesis that we have imposed on the ideals in $\mathcal{I}$, we see that $\ctr(R/I)$ is a finitely presented module of projective dimension 1 for each $I \in \mathcal{I}$, and whence $(\bigoplus_{I \in \mathcal{I}} \ctr(R/I))^\perp=\mathcal{T}$ is a 1-tilting class. Put $\mathcal{A}={}^\perp \mathcal{T}$. 

			Let $M_\mathcal{I}$ be the module from Definition~\ref{D01}. For each $n \in \omega$ let $M_n$ be the submodule of $M_\mathcal{I}$ generated by (the images of) all sequences in $\Lambda$ of length smaller then or equal to $n$. In particular, $M_0=\Span(w)$ is isomorphic to $R$. The quotient $M_{n+1}/M_n$ is generated by the cosets of all sequences in $\Lambda$ of length $n+1$. We claim that $M_{n+1}/M_n$ is isomorphic to a direct sum of a suitable number of copies of the modules $\ctr(R/I)$ with $I \in \mathcal{I}$. 
			
			In order to prove this claim, let us fix more notation: For each $n \in \omega$ denote by $\Lambda_n$ (resp. $\Lambda_{<n}$) a subset of $\Lambda$ consisting of all sequences of length $n$ (resp. smaller than $n$). Put $F_n=\Span(\Lambda_{<n+1}) \subseteq F$ and $G_n=F_n \cap G$, so that $M_n \simeq F_n/G_n$. Let $X=\{\lambda - \sum_{k \in n_I} x^I_k(\lambda\sqcup(I,k)) \mid \lambda \in \Lambda, I \in \mathcal{I}\}$ be the prescribed set of generators of $G$. Let $X_n = X \cap F_n$ for each $n \in \omega$. Observe that $X$ is actually a free basis of $G$, and furthermore, that $X \setminus X_n$ is linearly independent in $F$ modulo $F_n$. Indeed, our hypothesis of $\Hom_R(R/I,R)=0$ assures that $I$ has no non-trivial annihilator in $R$ for each $I \in \mathcal{I}$, and thus the elements of the form $\sum_{k \in n_I} x^I_k(\lambda\sqcup(I,k))$ are torsion-free, and hence they are linearly independent in $F/F_n$. It follows that $G_n=\Span(X_n)$, that is, $G_n$ is generated by elements
			$$\lambda - \sum_{k \in n_I} x^I_k(\lambda\sqcup(I,k)),$$
			where $\lambda \in \Lambda_{<n}, I \in \mathcal{I}$. From this it is easily seen that $M_{n+1}/M_n$ can be viewed as a module with generators $\Lambda_{n+1}$ and relations of the form 
			$$\sum_{k \in n_I} x^I_k(\lambda\sqcup(I,k))=0,$$ 
			where $\lambda \in \Lambda_n$ and $I \in \mathcal{I}$. It follows that $M_{n+1}/M_n \simeq \bigoplus_{\lambda \in \Lambda_n} \bigoplus_{I \in \mathcal{I}} \ctr(R/I)$, and the claim is proved.

			As $M_0=wR \simeq R$, we have that $M_\mathcal{I}$ is filtered by the set $\{R\} \cup \{\ctr(R/I) \mid I \in \mathcal{I} \} \subseteq \mathcal{A}$, and so $M_\mathcal{I} \in \mathcal{A}$, and also $M_\mathcal{I}/wR \in \mathcal{A}$. On the other hand, $M_\mathcal{I}$ is generated by $\Lambda$, and from the construction we have that for each $\lambda \in \Lambda$ and each $I \in \mathcal{I}$, $\lambda \in IM_\mathcal{I}$. It follows that $M_\mathcal{I} \in \mathcal{T}$.

			Altogether we have that the inclusion $R \simeq M_0 \rightarrow M_\mathcal{I}$ is a special $\mathcal{T}$-preenvelope of $R$. An argument \cite[Remark 13.47]{GT} then shows that $\delta_\mathcal{I}=M_\mathcal{I} \oplus M_\mathcal{I}/wR$ is a 1-tilting module in $\mathcal{A} \cap \mathcal{T}$, and thus generating the class $\mathcal{T}$.
	\end{proof}

	Combining Theorem~\ref{T01} and Proposition~\ref{P21} we obtain the following.

	\begin{thm}
			\label{T02}
			Let $R$ be a commutative ring. Then 
			$$\{\delta_\mathcal{G} \mid \mathcal{G} \text{ a faithful finitely generated Gabriel topology}\}$$
			is the set of representatives of equivalence classes of all 1-tilting modules over $R$.
	\end{thm}
	\subsection{An application}
	As an application, we present an alternative proof of the positive solution of the so-called Saorín's problem for commutative rings. The Saorín's problem is the following statement.

	\begin{prob}
			\label{P00}
			\emph{(\cite{SP})}
		Let $R$ be a ring and $T$ a 1-tilting module such that the induced torsion-free class $\mathcal{F}=\Ker \Hom_R(T,-)$ is closed under direct limits. Is then $T$ equivalent to a finitely generated 1-tilting module?
	\end{prob}

	The motivation of the problem is the recent result of Parra and Saorín \cite[Theorem 4.9]{SP}, which states that heart of the t-structure associated to a tilting torsion pair $(\mathcal{T},\mathcal{F})$ is a Grothendieck category if and only if $\mathcal{F}$ is closed under direct limits.

	If $R$ is commutative, then any finitely generated tilting module is projective, so a positive answer implies that $\mathcal{F}=\{0\}$. Problem~\ref{P00} has a negative answer in general, a very involved counter-example was found by Herzog, and further counter-examples that are non-commutative, but two-sided noetherian were constructed by Příhoda (\cite{PP}). On the other hand, Problem~\ref{P00} has a positive answer whenever $R$ is commutative, as proved by Bazzoni in \cite{PP}. We can now reprove the latter result in an elementary way using our classification of 1-tilting classes and 1-tilting modules.

	\begin{thm}\emph{(\cite{PP})}
		Let $R$ be a commutative ring and $T$ a 1-tilting module such that $\mathcal{F}=\Ker \Hom_R(T,-)$ is closed under direct limits. Then $T$ is projective.
	\end{thm}
	\begin{proof}
			By Theorem~\ref{T01} and Theorem~\ref{T02}, there is a faithful finitely generated Gabriel topology $\mathcal{G}$ such that $\mathcal{T}=T^\perp=\mathcal{G}\Div$ and we can without loss of generality assume that $T=\delta_\mathcal{G}$. Suppose that $T$ is not projective. Then necessarily $\mathcal{G}$ contains a non-trivial ideal. Let $\mathcal{I}$ be a basis of finitely generated ideals of $\mathcal{G}$. Let $M_\mathcal{I}$ be the module from the construction of $\delta_\mathcal{G}$, that is, $\delta_\mathcal{G}=M_\mathcal{I} \oplus M_\mathcal{I}/\Span(w)$, and $M_\mathcal{I}=\bigcup_{n \in \omega}M_n$ with $M_0=\Span(w) \simeq R$ and $M_{n+1}/M_n$ isomorphic to a direct sum of copies of $\ctr(R/I), I \in \mathcal{I}$ for each $n \in \omega$.

			Denote by $t(-)$ the torsion radical of the torsion pair $(\mathcal{T},\mathcal{F})$. Since $\mathcal{F}$ is closed under direct limits, the direct limit functor is exact, and $M \in \mathcal{T}$, we have $M_\mathcal{I}=\bigcup_{n \in \omega}t(M_n)$. It follows that there is $n \in \omega$ such that $w \in t(M_n)$. In other words, there is a submodule $X$ of $M_n$ containing $w$ such that $X=IX$ for each $I \in \mathcal{I}$. It is clear that $n \neq 0$, since then it would follow that $R=t(R)$, which cannot be the case since $\mathcal{G}$ contains non-trivial ideals. 
			
			Suppose that $n>0$. From now on we adopt the notation of Definition~\ref{D01} for the generators of $M_n$. For each $n$-tuple of ideals $\bar{I}=(I_1,I_2,\ldots,I_n) \in \mathcal{I}^n$, we define a finite subset $Y_{\bar{I}}$ of $M_n$ as follows: Let $Y_{\bar{I},1}=\{(w,(I_1,k)) \mid k \in n_{I_1}\}$. For $1<j\leq n$, we put $Y_{\bar{I},j}=\{\lambda \sqcup (I_j,k) \mid \lambda \in Y_{\bar{I},j-1}, k \in n_{I_j}\}$. Finally, we set $Y_{\bar{I}}=\bigcup_{1 \leq j \leq n} Y_{\bar{I},j}$. Note that $\Span(Y_{\bar{I}})$ is a free $R$-module with basis consisting of sequences from $Y_{\bar{I}}$ of maximal length (that is, of length $n$). We can index this basis as follows: Denote by $\kappa_{\bar{I}}$ the set of all sequences $\bar{k}=(k_1,k_2,\ldots,k_n) \in \omega^n$ such that $k_j \in n_{I_k}$ for all $1 \leq j \leq n$. For each $\bar{k} \in \kappa_{\bar{I}}$ let $\lambda_{\bar{k}}^{\bar{I}}$ denote the element $w \sqcup (I_1,k_1) \sqcup (I_2,k_2) \sqcup \ldots \sqcup (I_n,k_n)$. Then $\Span(Y_{\bar{I}})$ is a free module with basis $\{\lambda_{\bar{k}}^{\bar{I}} \mid \bar{k} \in \kappa_{\bar{I}}\}$. Also, note that $w=\sum_{\bar{k} \in \kappa_{\bar{I}}} x^{I_1}_{k_1}x^{I_2}_{k_2}\cdots x^{I_n}_{k_n}\lambda_{\bar{k}}^{\bar{I}}$. Denote $x_{\bar{k}}^{\bar{I}}=x^{I_1}_{k_1}x^{I_2}_{k_2}\cdots x^{I_n}_{k_n}$ for each $\bar{I} \in \mathcal{I}^n$ and $\bar{k} \in \kappa_{\bar{I}}$. Finally, it is easy to see that $M_n=\Span(Y_{\bar{I}} \mid \bar{I} \in \mathcal{I}^n)$.
			
	As $X$ is the direct limit of all its finitely generated submodules containing $w$, we again use the hypothesis of $\mathcal{F}$ being closed under direct limits in order to find a finitely generated submodule $N$ of $X$ such that $w \in t(N)$. As $N$ is finitely generated, there are $\bar{I}^1,\bar{I}^2,\ldots,\bar{I}^m \in \mathcal{I}^m$ such that $t(N) \subseteq \Span(Y_{\bar{I}^1},Y_{\bar{I}^2},\ldots,Y_{\bar{I}^m})$. Denote this span by $S$. Then $S$ is a module with generators $\{\lambda_{\bar{k}}^{\bar{I}^j} \mid 1 \leq j \leq m, \bar{k} \in \kappa_{\bar{I}^j}\}$ subject to the following relations:
	\begin{equation}
			\label{E05}
	w=\sum_{\bar{k} \in \kappa_{\bar{I}^1}} x^{\bar{I}^1}_{\bar{k}}\lambda_{\bar{k}}^{\bar{I^1}}=\sum_{\bar{k} \in \kappa_{\bar{I}^2}} x^{\bar{I}^2}_{\bar{k}}\lambda_{\bar{k}}^{\bar{I^2}}=\cdots=\sum_{\bar{k} \in \kappa_{\bar{I}^m}} x^{\bar{I}^m}_{\bar{k}}\lambda_{\bar{k}}^{\bar{I^m}}.
	\end{equation}

	This leads to a contradiction. Indeed, since $t(N)$ is divisible by each ideal in $\mathcal{I}$, and $\mathcal{I}$ consists of finitely generated faithful (and therefore not idempotent) ideals, we infer that there is an ideal $J \in \mathcal{I}$ such that $J \subsetneq \prod_{1 \leq i \leq m, 1 \leq j \leq n}I^i_j$. Hence, there are elements $s^i_{\bar{k}} \in J$ for each $1 \leq i \leq m$ and $\bar{k} \in \kappa_{\bar{I}^i}$, such that $w=\sum_{i=1}^m \sum_{k \in \kappa_{\bar{I}^i}} s^i_{\bar{k}} \lambda_{\bar{k}}^{\bar{I}^i}$. Note that $S/\Span(w)$ decomposes as follows: $$S/\Span(w) \simeq \bigoplus_{1 \leq i \leq m}R^{(\kappa_{\bar{I}^i})}/\Span(\sum_{\bar{k}\in\kappa_{\bar{I}^i}}x_{\bar{k}}^{\bar{I}^i}\bar{k}).$$ Projecting $S$ onto the $i$-th summand in this decomposition yields that there is $t_i \in R$ such that $s^i_{\bar{k}}=t_i x^{\bar{I}^i}_{\bar{k}}$ for all $1 \leq i \leq m$ and $\bar{k} \in \kappa_{\bar{I}^i}$. Since $\Span(x^{\bar{I}^i}_{\bar{k}} \mid \bar{k} \in \kappa_{\bar{I}^i})=I_1^i I_2^i \cdots I_n^i$, we infer that $t_i \in (I_1^i I_2^i \cdots I_n^{i} : J)$. Using the relations (\ref{E05}) several times, we get that $w=(t_1+t_2+\cdots+t_m)\sum_{\bar{k} \in \kappa_{\bar{I}^1}}x_{\bar{k}}^{\bar{I}^1}\lambda_{\bar{k}}^{\bar{I}^1}$. Since $\ann(w)=0$, this implies that $t_1+t_2+\cdots+t_m=1$. But $t_i \in (I_1^i I_2^i \cdots I_n^i : J) \subseteq (\prod_{1 \leq i' \leq m, 1 \leq j \leq n}I_{j}^{i'} : J) \neq R$ for each $i=1,\ldots,m$ by the assumption on $J$, making the assertion $t_1+t_2+\cdots+t_j=1$ a contradiction.
\end{proof}

\section{Perfect localizations}
	As hereditary torsion classes coincide with localizing subcategories of $\ModR$, each hereditary torsion class $\mathcal{E}$ in $\ModR$ gives rise to a (Serre) localization $\ModR \rightarrow \ModR/\mathcal{E}$. Therefore, each 1-tilting class over a commutative ring corresponds naturally to some localization functor. The localized category is not in general a module category, and so it is not induced by a ring homomorphism. In this section, we focus on the case when this localization is induced by a flat ring epimorphism. In particular, we describe when this so-called \emph{perfect localization} allows to replace the Fuchs-Salce module by a much nicer tilting module, arisen from a ring of quotients.

	Given a Gabriel topology $\mathcal{G}$, recall that a module $M$ is $\mathcal{G}$\emph{-closed} if the inclusion $I \subseteq R$ induces an isomorphism $\Hom_R(R,M) \rightarrow \Hom_R(I,M)$ for any ideal $I \in \mathcal{G}$. Denote the full subcategory of all $\mathcal{G}$-closed modules by $\mathcal{X}(\mathcal{G})$. This subcategory is \emph{Giraud}, that is, a full subcategory of $\ModR$ such that its inclusion into $\ModR$ has a left adjoint which is exact (and, in fact, all Giraud subcategories of $\ModR$ are of form $\mathcal{X}(\mathcal{G})$ for some Gabriel topology $\mathcal{G}$). The composition of this left adjoint and the original inclusion yields a \emph{localization functor} $L: \ModR \rightarrow \ModR$. The unit of the adjunction $\eta_R: R \rightarrow L(R)$ then induces a ring structure on $Q_\mathcal{G}=L(R)$, with the unit $\eta_R$ being a ring homomorphism. For all details we refer to Chapters VII.-XI. in \cite{BS}, as for the proof of the following:
	\begin{thm}
		(\cite[XI, Proposition 3.4]{BS})
		\label{T99}
		Let $R$ be a commutative ring, and $\mathcal{G}$ a Gabriel topology. Then the following are equivalent:
		\begin{enumerate}
				\item $\eta_R: R \rightarrow Q_\mathcal{G}$ is a flat ring epimorphism, and $\{I \subseteq R \mid Q_\mathcal{G}=IQ_\mathcal{G}\}=\mathcal{G}$,
				\item $\mathcal{X}(\mathcal{G})$ is naturally equivalent to $\operatorname{Mod-Q_\mathcal{G}}$,
			\item the $R$-module $Q_\mathcal{G}$ is $\mathcal{G}$-divisible.
		\end{enumerate}
	\end{thm}
	If $\mathcal{G}$ satisfies conditions of this theorem, we call it \emph{perfect}. Say that a perfect localization $\lambda: R \rightarrow S$ is \emph{faithful}, if the map $\lambda$ is injective. Say that two ring epimorphisms $\lambda:R \rightarrow S, \lambda':R \rightarrow S'$ are \emph{equivalent} if there is a ring isomorphism $\varphi: S \rightarrow S'$, such that $\lambda'=\varphi\lambda$. The equivalence classes of ring epimorphism under this equivalence are called \emph{epiclasses} of $R$. By \cite[XI, Theorem 2.1]{BS}, the ring maps $R \rightarrow Q_\mathcal{G}$, with $\mathcal{G}$ running through perfect Gabriel topologies, parametrize all epiclasses of flat ring epimorphisms, which justifies the terminology \emph{perfect localization} instead of flat ring epimorphism.
	
	We recall that a ring is right \emph{semihereditary}, if any finitely generated right ideal is projective.
	\begin{thm}
			\emph{(cf. \cite[Proposition 7.4]{SMA})}
			Let $R$ be a commutative semihereditary ring, $\mathcal{T}$ a 1-tilting class, and $\mathcal{G}$ a Gabriel topology associated to this class (via Theorem~\ref{T01}). Then $\mathcal{G}$ is perfect, and the perfect localization $\eta: R \rightarrow Q_\mathcal{G}$ is faithful. Furthermore, there is a 1-1 correspondence between $1$-tilting classes $\mathcal{T}$ and epiclasses of faithful perfect localizations $R \xhookrightarrow{} S$; the correspondence given by
			$$\Gamma: \mathcal{T}=\mathcal{G}\Div \mapsto (R \xhookrightarrow{} Q_\mathcal{G}),$$
			$$\Delta: (R \xhookrightarrow{} S) \mapsto \{I \subseteq R \mid S=IS\}\Div.$$
	\end{thm}
	\begin{proof}
			By Theorem~\ref{T01}, the Gabriel topology $\mathcal{G}$ is necessarily finitely generated and faithful, and hence perfect by \cite[XI, Corollary 3.5]{BS} and \cite[IX, Proposition 5.2]{BS}, and any perfect Gabriel topology inducing a faithful perfect localization arises in this way. The map $\eta_R: R \rightarrow Q_\mathcal{G}$ is injective again by faithfulness of $\mathcal{G}$ and \cite[IX, Lemma 1.2]{BS}. Together with Theorem~\ref{T99}, this shows that $\Gamma$ is well-defined. By Theorem~\ref{T99} and \cite[XI, Theorem 2.1]{BS}, $\Delta(R \xhookrightarrow{} S)$ is equal to some perfect Gabriel topology $\mathcal{G}$, which is finitely generated by \cite[XI, Proposition 3.4]{BS}, and faithful by \cite[IX, Lemma 1.2]{BS}, and thus $\Delta$ is well-defined by Theorem~\ref{T01}. Finally, $\Gamma$ and $\Delta$ are mutually inverse, for checking which it is now enough to use the fact that epiclasses of faithful perfect localizations are parametrized by the set $\{R \rightarrow Q_\mathcal{G} \mid \mathcal{G} \text{ a perfect faithful Gabriel topology}\}$.
	\end{proof}

\begin{definition}
		We say that a tilting module $T$ \emph{arises from a perfect localization}, if there is an faithful perfect localization $R \xhookrightarrow{} S$ such that $S \oplus S/R$ is a 1-tilting module equivalent to $T$.
\end{definition}
We can now prove the following generalization of (part of) \cite[Theorem 4.10]{LA} and \cite[Theorem 1.1]{AHT}.

\begin{thm}
	Let $R$ be a commutative ring, $T$ a 1-tilting module, and $\mathcal{G}$ a Gabriel topology associated to $\mathcal{T}=T^\perp$ in the sense of Theorem~\ref{T01}. Then the following are equivalent:
\begin{enumerate}
	\item $\mathcal{G}$ is perfect, and $\pd_R Q_\mathcal{G} \leq 1$,
	\item $T$ arises from a perfect localization,
	\item $\gen(Q_\mathcal{G})=\mathcal{G}\Div$.
\end{enumerate}
\begin{proof}
		(1) $\Rightarrow$ (2): By \cite[Lemma 1.10]{LA}, the module $T'=Q_\mathcal{G} \oplus Q_\mathcal{G}/R$ is a 1-tilting module. Therefore, there is by Theorem~\ref{T01} a faithful finitely generated Gabriel topology $\mathcal{G}'$ such that $T'^\perp=\mathcal{G}'\Div$. Using Theorem~\ref{T99}, we conclude that $\mathcal{G}=\mathcal{G}'$, proving that $\mathcal{T}=T^\perp$, and therefore $T$ is equivalent to $T'$.

		(2) $\Rightarrow$ (3): Follows quickly from $\gen(Q_\mathcal{G})=\gen(Q_\mathcal{G}\oplus Q_\mathcal{G}/R)=\mathcal{T}=\mathcal{G}\Div$.

		(3) $\Rightarrow$ (1): That $\mathcal{G}$ is perfect follows directly from Theorem~\ref{T99}. By (3), there is an epimorphism $Q_\mathcal{G}^{(X)} \rightarrow \delta_\mathcal{G}$ for some set $X$. Since $Q_\mathcal{G} \in \mathcal{T}$, there is also an epimorphism $\delta_\mathcal{G}^{(Y)} \rightarrow Q_\mathcal{G}$ for some set $Y$. Together we get an epimorphism $Q_\mathcal{G}^{(X \times Y)} \rightarrow Q_\mathcal{G}$ in $\ModR$. As $R \rightarrow Q_\mathcal{G}$ is a ring epimorphism of $R$, $\operatorname{Mod-Q_\mathcal{G}}$ is a full subcategory of $\ModR$, and thus the epimorphism from last sentence is actually a map of $Q_\mathcal{G}$-modules, and hence it splits. But then also the epimorphism $Q_\mathcal{G}^{(X \times Y)} \rightarrow \delta_\mathcal{G}^{(Y)}$ splits, and thus $Q_\mathcal{G}$ has projective dimension at most 1 over $R$.
\end{proof}
\end{thm}
\section*{Acknowledgement}
I am grateful to Jan Trlifaj for many fruitful discussions, as well as for careful reading of my manuscripts. I would like to thank Jan Šťovíček for our discussions, and for introducing me to the work of Thomason.

	\bibliographystyle{alpha}
	\bibliography{tilt_comm}

\end{document}